\documentclass[12pt]{amsart}

\usepackage{amssymb, amsmath, amsthm, amsfonts}
\usepackage{mathrsfs,comment}
\usepackage{graphicx}
\usepackage{placeins}
\usepackage{todonotes}
\usepackage{rotating}
\usepackage{tikz}
\usepackage{float}
\usepackage{subfigure}
\usepackage{hvfloat}
\usepackage{caption}
\usepackage{pdflscape}
\usepackage{hyperref}  
\usepackage{url}
\usepackage[all,arc,2cell]{xy}
\UseAllTwocells
\usepackage{enumerate}
\usepackage{chngcntr}
 \usepackage{lineno}
 \usepackage{blindtext}
\usepackage{verbatim}
\usepackage{soul}
\usepackage{sseq}
\usepackage{mathtools}
\usepackage{times}
\usepackage{tikz-cd}
\usepackage{pgfplots}
\usepackage{pgfplotstable}

\usepackage[normalem]{ulem}
\newcommand{\stkout}[1]{\ifmmode\text{\sout{\ensuremath{#1}}}\else\sout{#1}\fi}

    \hypersetup{%
    bookmarksnumbered=true,%
    bookmarks=true,%
    colorlinks=true,%
    linkcolor=blue,%
    citecolor=blue,%
    filecolor=blue,%
    menucolor=blue,%
    pagecolor=blue,%
    urlcolor=blue,%
    pdfnewwindow=true,%
    pdfstartview=FitBH}

\def\@url#1{{\tt\def~{\lower3.5pt\hbox{\char'176}}\def\_{\char'137}#1}}

\let\fullref\autoref
\def\makeautorefname#1#2{\expandafter\def\csname#1autorefname\endcsname{#2}}
\makeautorefname{equation}{Equation}%
\makeautorefname{footnote}{footnote}%
\makeautorefname{item}{item}%
\makeautorefname{figure}{Figure}%
\makeautorefname{table}{Table}%
\makeautorefname{part}{Part}%
\makeautorefname{appendix}{Appendix}%
\makeautorefname{chapter}{Chapter}%
\makeautorefname{section}{Section}%
\makeautorefname{subsection}{Section}%
\makeautorefname{subsubsection}{Section}%
\makeautorefname{paragraph}{Paragraph}%
\makeautorefname{subparagraph}{Paragraph}%
\makeautorefname{theorem}{Theorem}%
\makeautorefname{theo}{Theorem}%
\makeautorefname{thm}{Theorem}%
\makeautorefname{addendum}{Addendum}%
\makeautorefname{addend}{Addendum}%
\makeautorefname{add}{Addendum}%
\makeautorefname{maintheorem}{Main theorem}%
\makeautorefname{mainthm}{Main theorem}%
\makeautorefname{corollary}{Corollary}%
\makeautorefname{claim}{Claim}%
\makeautorefname{corol}{Corollary}%
\makeautorefname{coro}{Corollary}%
\makeautorefname{cor}{Corollary}%
\makeautorefname{lemma}{Lemma}%
\makeautorefname{lemm}{Lemma}%
\makeautorefname{lem}{Lemma}%
\makeautorefname{sublemma}{Sublemma}%
\makeautorefname{sublem}{Sublemma}%
\makeautorefname{subl}{Sublemma}%
\makeautorefname{proposition}{Proposition}%
\makeautorefname{proposit}{Proposition}%
\makeautorefname{propos}{Proposition}%
\makeautorefname{propo}{Proposition}%
\makeautorefname{prop}{Proposition}%
\makeautorefname{property}{Property}
\makeautorefname{proper}{Property}
\makeautorefname{scholium}{Scholium}%
\makeautorefname{step}{Step}%
\makeautorefname{conjecture}{Conjecture}%
\makeautorefname{conject}{Conjecture}%
\makeautorefname{conj}{Conjecture}%
\makeautorefname{question}{Question}
\makeautorefname{questn}{Question}
\makeautorefname{quest}{Question}
\makeautorefname{ques}{Question}
\makeautorefname{qn}{Question}
\makeautorefname{definition}{Definition}%
\makeautorefname{defin}{Definition}%
\makeautorefname{defi}{Definition}%
\makeautorefname{def}{Definition}%
\makeautorefname{defn}{Definition}%
\makeautorefname{dfn}{Definition}%
\makeautorefname{notation}{Notation}
\makeautorefname{nota}{Notation}
\makeautorefname{notn}{Notation}
\makeautorefname{remark}{Remark}%
\makeautorefname{rema}{Remark}%
\makeautorefname{rem}{Remark}%
\makeautorefname{rmk}{Remark}%
\makeautorefname{rk}{Remark}%
\makeautorefname{remarks}{Remarks}%
\makeautorefname{rems}{Remarks}%
\makeautorefname{rmks}{Remarks}%
\makeautorefname{rks}{Remarks}%
\makeautorefname{example}{Example}%
\makeautorefname{examp}{Example}%
\makeautorefname{exmp}{Example}%
\makeautorefname{exmps}{Examples}%
\makeautorefname{exam}{Example}%
\makeautorefname{exa}{Example}%
\makeautorefname{algorithm}{Algorith}%
\makeautorefname{algo}{Algorith}%
\makeautorefname{alg}{Algorith}%
\makeautorefname{axiom}{Axiom}%
\makeautorefname{axi}{Axiom}%
\makeautorefname{ax}{Axiom}%
\makeautorefname{case}{Case}%
\makeautorefname{claim}{Claim}%
\makeautorefname{clm}{Claim}%
\makeautorefname{assumption}{Assumption}%
\makeautorefname{assumpt}{Assumption}%
\makeautorefname{conclusion}{Conclusion}%
\makeautorefname{concl}{Conclusion}%
\makeautorefname{conc}{Conclusion}%
\makeautorefname{condition}{Condition}%
\makeautorefname{condit}{Condition}%
\makeautorefname{cond}{Condition}%
\makeautorefname{construction}{Construction}%
\makeautorefname{construct}{Construction}%
\makeautorefname{const}{Construction}%
\makeautorefname{cons}{Construction}%
\makeautorefname{criterion}{Criterion}%
\makeautorefname{criter}{Criterion}%
\makeautorefname{crit}{Criterion}%
\makeautorefname{exercise}{Exercise}%
\makeautorefname{exer}{Exercise}%
\makeautorefname{exe}{Exercise}%
\makeautorefname{problem}{Problem}%
\makeautorefname{problm}{Problem}%
\makeautorefname{probm}{Problem}%
\makeautorefname{prob}{Problem}%
\makeautorefname{solution}{Solution}%
\makeautorefname{soln}{Solution}%
\makeautorefname{sol}{Solution}%
\makeautorefname{summary}{Summary}%
\makeautorefname{summ}{Summary}%
\makeautorefname{sum}{Summary}%
\makeautorefname{operation}{Operation}%
\makeautorefname{oper}{Operation}%
\makeautorefname{observation}{Observation}%
\makeautorefname{observn}{Observation}%
\makeautorefname{obser}{Observation}%
\makeautorefname{obs}{Observation}%
\makeautorefname{ob}{Observation}%
\makeautorefname{convention}{Convention}%
\makeautorefname{convent}{Convention}%
\makeautorefname{conv}{Convention}%
\makeautorefname{cvn}{Convention}%
\makeautorefname{warning}{Warning}%
\makeautorefname{warn}{Warning}%
\makeautorefname{note}{Note}%
\makeautorefname{fact}{Fact}%
\makeautorefname{thmbig}{Theorem}%
\makeautorefname{conjbig}{Conjecture}%

  \makeatletter
                   \let\c@lemma\c@theorem
                  \makeatother

\newtheorem{thm}{Theorem}[section]

\newtheorem{cor}{Corollary}[section]
\newtheorem{prop}{Proposition}[section]
\newtheorem{lem}{Lemma}[section]

\newtheorem*{thmbig}{Theorem}

\theoremstyle{definition}

\newtheorem{rem}{Remark}[section]

\makeatletter
\let\c@lem=\c@thm
\let\c@theorem=\c@thm
\let\c@notation=\c@thm
\let\c@cor=\c@thm
\let\c@prop=\c@thm
\let\c@lem=\c@thm
\let\c@defn=\c@thm
\let\c@exmps=\c@thm
\let\c@rem=\c@thm
\let\c@warn=\c@thm
\let\c@claim=\c@thm
\let\c@conj=\c@thm
\let\c@quest=\c@thm
\makeatother

\numberwithin{equation}{section}

\def\quickop#1{\expandafter\newcommand\csname #1\endcsname{\operatorname{#1}}}
\quickop{Hom} \quickop{End} \quickop{Aut} \quickop{Tel} \quickop{Mic} \quickop{map}
\quickop{Ext} \quickop{Tor} \quickop{Cotor} \quickop{Id} \quickop{Coker} \quickop{Ker}
\quickop{Lim} \quickop{Colim} \quickop{Holim} \quickop{Hocolim}
\quickop{id} \quickop{tel} \quickop{mic} \quickop{coker}
\quickop{colim} \quickop{holim} \quickop{hocolim} \quickop{im} \quickop{Syl} \quickop{Ind}

\newcommand{\F}{\mathbb{F}}

\newcommand{\G}{\mathbb{G}}
\newcommand{\W}{\mathbb{W}}
\newcommand{\Q}{\mathbb{Q}}
\newcommand{\smsh}{\wedge}

\newcommand{\xra}{\xrightarrow}

\newcommand{\ZZ}{\mathbb{Z}}
\newcommand{\FF}{\mathbb{F}}

\DeclareFontFamily{OMS}{rsfs}{\skewchar\font'60}
\DeclareFontShape{OMS}{rsfs}{m}{n}{<-5>rsfs5 <5-7>rsfs7 <7->rsfs10 }{}
\DeclareSymbolFont{rsfs}{OMS}{rsfs}{m}{n}
\DeclareSymbolFontAlphabet{\scr}{rsfs}

\def\makeop#1{\expandafter\def\csname #1\endcsname{\mathop{\mathrm{#1}}\nolimits}}

\makeop{Gal}
\makeop{id}
\makeop{Mod}
\makeop{Tot}
\makeop{gr}
\makeop{Out}
\makeop{Hom}
\makeop{Ext}
\makeop{End}
\makeop{Aut}
\makeop{Tor}
\makeop{ev}
\makeop{Sym}
\def\FF{\mathbb{F}}

\def\GG{\mathbb{G}}
\def\SS{\mathbb{S}}
\def\WW{{{\mathbb{W}}}}
\def\ZZ{{{\mathbb{Z}}}}
\def\Z{{{\mathbb{Z}}}}
\def\Ext{\mathrm{Ext}}

\definecolor{darkspringgreen}{rgb}{0.09, 0.45, 0.27}

\newcommand{\la}{\Delta}  
\newcommand{\lb}{b} 
\newcommand{\lc}{\overline{b}} 
\newcommand{\ld}{\overline{\Delta}} 

\newcommand{\bock}[1]{\delta^{(#1)}}

\DeclareRobustCommand\circled[1]{\tikz[baseline=(char.base)]{
   \node[shape=circle,draw,inner sep=0pt] (char) {#1};}}

\usepackage{fullpage}

\newcommand{\cvone}{(v_1^2)}
\newcommand{\cvonevtwo}{(v_1v_2)}
\newcommand{\cvtwo}{(v_2^2)}
\newcommand{\cj}{(u_1^3)}

\title{The $\alpha$-Family in the $K(2)$-Local Sphere at the Prime $2$}
\date{\today}

\author[A. Beaudry]{Agn\`es Beaudry}
\address{Department of Mathematics\\ University of Colorado at Boulder \\ \newline Campus Box 395 \\ Boulder \\ Colorado \\ 80309}

\thanks{This material is based upon work supported by the National Science Foundation under Grant No.~DMS-1725563.}

\begin{document}
\maketitle
\begin{abstract}
In this note, we compute the image of the $\alpha$-family in the homotopy of the $K(2)$-local sphere at the prime $p=2$ by locating its image in the algebraic duality spectral sequence. This is a stepping stone for the computation of the homotopy groups of the $K(2)$-local sphere at the prime $2$ using the duality spectral sequences.
\end{abstract}
\setcounter{tocdepth}{1}
\tableofcontents

\subsection*{Acknowledgements}
This note was born in conversations which happened during the writing of \cite{BGH} and the author is indebted to her collaborators, Paul Goerss and Hans-Werner Henn. Many of the methods and ideas used here are borrowed from that rich collaboration. She also thanks Zhouli Xu for useful conversations related to this topic. 

Mark Mahowald knew how the $\alpha$-family would be detected in the duality spectral sequences and this paper makes his sketches precise. Some of the computations used in the proof of this theorem are also closely related to results of Mahowald and Rezk in \cite{MR}.

\section{Introduction}

The first periodic family in the homotopy groups of spheres was constructed by Adams in his study of the image of the $J$ homomorphism, which culminated in what is now one of the must-read articles in algebraic topology, \emph{On the Groups $J(X)$ -- IX} \cite{AdamsJIX}. In the last section of this paper, Adams uses self-maps of Moore spaces to construct elements of the homotopy groups of spheres that he denotes by $\alpha$. These elements are intimately related to $K$-theory and are part of what is now called the ``$\alpha$-family''. 

The $\alpha$-family is one of the few computable families of elements in the stable homotopy groups of spheres. It is the first of its kind and, with its successors the $\beta$ and $\gamma$-families, it now belongs to a collection of classes known as the Greek-letter elements. In their cornerstone paper on periodicity in the Adams-Novikov Spectral Sequence, Miller, Ravenel and Wilson \cite{MRW} give an intimate connection between the Greek-letter elements and the chromatic spectral sequence, and establish the importance of the chromatic point of view for computations of the homotopy groups of spheres. 

Chromatic homotopy as it is known today comes from Morava's insight that there should be higher analogs of $p$-completed $K$-theory. They should carry higher Adams operations, and detect periodic families which are generalizations of the image of $J$. These cohomology theories are called the Morava $E$-theories $E_n$ and the associated mod $p$ theories are called the Morava $K$-theories $K(n)$. The higher Adams operations form a group called the Morava stabilizer group, denoted $\GG_n$. The theories $E_n$ and $K(n)$ are complex oriented ring spectra whose construction is based on the deformation theory of height $n$ formal groups. 

The Morava $E$ and $K$-theories detect periodic families of elements in the homotopy groups of spheres. There are various ways to make this precise. One is through the eyes of Bousfield localization. The Chromatic Convergence Theorem of Hopkins and Ravenel states that the $p$-local sphere spectrum $S_{(p)}$ is the (homotopy) inverse limit of the Bousfield localizations $S_{E_n}$ of the sphere at the Morava $E$-theories. One then studies $S_{(p)}$ through its images under the natural maps $S_{(p)} \to S_{E_n}$. Further, the $S_{E_n}$ can be inductively reassembled from the localizations at the Morava $K$-theories via a homotopy pull-back
\[\xymatrix{ S_{E_n} \ar[r]  \ar[d] & S_{K(n)} \ar[d]  \\
 S_{E_{n-1}} \ar[r]  & (S_{K(n)})_{E_{n-1}} .}\]

These facts highlight the importance of computing  $\pi_*S_{E_n}$ and $\pi_*S_{K(n)}$.  The standard tools for computing these homotopy groups are two closely related spectral sequences. Note that the $E_n$-local sphere is equivalent to $S_{E(n)}$, where $E(n)$ is the Johnson-Wilson spectrum, a ``thiner'' version of $E_n$. The $E(n)$-Adams-Novikov Spectral Sequence computes the homotopy groups of $S_{E(n)} \simeq S_{E_n}$:
\[\Ext_{E(n)_*E(n)}^{*,*}(E(n)_*, E(n)_*) \Longrightarrow \pi_{*}S_{E(n)} .\]
The second spectral sequence is the $K(n)$-local $E_n$-Adams-Novikov Spectral Sequence, which computes the homotopy groups of $S_{K(n)} \simeq E_n^{h\G_n}$. Its $E_2$-term can be identified with continuous cohomology groups:
\[H^*(\G_n, (E_n)_*) \Longrightarrow \pi_{*}S_{K(n)} .\]

We give an overview of what is known. First $S_{K(0)}$ and $S_{E_0}$ are both the rational sphere $S_{\Q}$. The computation of $\pi_*S_{K(1)}$ and $\pi_*S_{E_1}$ can be obtained from the classical computations of Adams, Atiyah and others on the image of $J$ and the action of the Adams operations. The computation of $\pi_*S_{E_2}$ and $\pi_*S_{K(2)}$ are entirely different beasts. Shimomura, Wang and Yabe have done extensive work on computing these homotopy groups at various primes. The case $p\geq 5$ is treated in \cite{shimyab} and is also nicely presented in \cite{behse2}. The case $p=3$ is treated in \cite{shimwangp3, shimmoore3} and the case $p=2$ is partially treated in \cite{shimwangp2, shimmoore2}.

The height two computations are extremely difficult and the answers contain an enormous amount of information that is hard to interpret and analyze. Having multiple point of views seems to have become an imperative for our understanding of chromatic height two phenomena. 

In \cite{ghmr}, Goerss, Henn, Mahowald and Rezk establish a different approach to height two computations. It relies on resolutions of the $K(2)$-local sphere called the \emph{duality resolutions}, from which one obtains various spectral sequences. For certain subgroups $G$ of $\GG_2$, the \emph{topological duality spectral sequences} converge to $\pi_*E_2^{hG}$ and the \emph{algebraic duality spectral sequences} converge to $H^*(G, (E_2)_*)$.

The advantage of the duality spectral sequences is that they organize the computations and the answers in a systematic way.  For $p\geq 5$, these methods are used in 
\cite{lader}, for $p= 3$, in \cite{HKM} and for $p=2$, in \cite{beaudrytowards} to perform computations for the $K(2)$-local Moore spectrum. The homotopy of $\pi_*S_{K(2)}$ at $p=3$ has been analyzed by Goerss, Henn, Karamanov, Mahowald using duality methods, but has not been fully recorded yet. 

Duality spectral sequence techniques are also being used to solve other central problems in chromatic homotopy theory. They have been crucial in the study of the Chromatic Splitting Conjecture \cite{GoerssSplit, BGH} at $p=2$ and $p=3$. In particular, they play a central role in the disproof of the strongest form of the conjecture at $p=2$ \cite{BGH}.
The computations of the $K(2)$-local Picard groups and of the Gross-Hopkins dual of the sphere at the prime $p=3$ rely on the duality spectral sequences  \cite{GHMRPIC,GHBC}. These are currently being adapted by the author and her collaborators to solve the same problems at $p=2$. Finally, Bhattacharya and Egger use the duality techniques to compute the homotopy groups of the first example of a type $2$ complex with a $v_2^1$ self-map \cite{BEZ}.

The current paper is concerned with computations of $\pi_*S_{K(2)}$ at $p=2$ using duality spectral sequence techniques and we finish the introduction by stating our result. When computing $\pi_*S_{K(2)}$, a first and essential step is to locate the $\alpha$-family in the computation. The goal of this paper is to do this at $p=2$, using the duality techniques. The results in this paper are a stepping stone for a full computation of $\pi_*S_{K(2)}$ using the duality spectral sequences. We will recall the precise definition of the $\alpha$-family in \fullref{sec:alphafamiliy}. We will define the algebraic duality spectral sequence and the subgroup $\SS_2^1 \subseteq \GG_2$ in \fullref{sec:tools}. Our main results are summarized in the following statement.

\begin{thmbig}
Let $p=2$. The elements $\alpha_{i/j} \in \Ext_{BP_*BP}^{1,2i}(BP_*,BP_*)$ map non-trivially to $H^1(\SS_2^1, (E_2)_{2i})$. 
In the algebraic duality spectral sequence
\[E_1^{p,q,t} = H^q(F_p, (E_2)_t) \Longrightarrow H^{p+q}(\SS_2^1, (E_2)_t) \] 
the $\alpha$s are detected as follows:
\begin{enumerate}[(a)]
\item $\alpha_{2/2} \in E_1^{0,1,4}$
\item $\alpha_{i/1} \in E_1^{0,1,2i}$ if $i\geq 1$ is odd.
\item $\alpha_{i/j} \in E_1^{1,0,2i}$ if $i$ is even.
\end{enumerate}
The maps
\[H^1(\GG_2, E_{t}) \to  H^1(\SS_2^1, (E_2)_{t})\]
in degrees $t\neq 0$ are injective so that the image of the $\alpha_{i/j}$ have unique lifts in $H^1(\GG_2, E_{2i})$. 

In the spectral sequence
\[H^*(\G_2, (E_2)_*) \Longrightarrow \pi_{*}S_{K(2)} \]
the $\alpha$-family supports the standard pattern of differentials and the family of elements detected by the $\alpha$s in $\pi_*S$ maps non-trivially to $\pi_*S_{K(2)}$. The same holds in $\pi_*E_2^{h\SS_2^1}$ and the associated homotopy fixed point spectral sequence.
\end{thmbig}

\section{The $\alpha$-family in the Adams-Novikov Spectral Sequence}\label{sec:alphafamiliy}
In this section, we review the construction of the $\alpha$-family and fix notation. We let $S = S_{(2)}$ and $BP$ be the $2$-local Brown-Peterson spectrum.
The Adams-Novikov Spectral Sequence is given by
\[E_2^{s,t} = \Ext_{BP_*BP}^{s,t} (BP_*, BP_*X) \Longrightarrow  \pi_{t-s} X_{(2)}.\]
The $\alpha$-family is a collection of elements $\alpha_{i/j} \in \Ext_{BP_*BP}^{1,*}(BP_*, BP_*)$ which we construct below. 

\begin{rem}We also call the collection of non-trivial elements of $\pi_*S$ detected by the $\alpha$s the $\alpha$-family, or the \emph{topological} $\alpha$-family when we wish to make the distinction clear.
\end{rem}

To define the $\alpha$-family, one first shows that there is an isomorphism
\[ \Ext^{0,\ast}_{BP_\ast BP}(BP_\ast,BP_\ast/2) \cong \FF_2[v_1]. \]
See for example Theorem 4.3.2 of \cite{ravgreen}.
The $\alpha$-family in $\Ext^{1,\ast}_{BP_\ast BP}(BP_\ast,BP_\ast)$ is defined by taking the image of
the powers of $v_1$ under various Bockstein homomorphisms. 
Define
$x_{1,n} \in v_1^{-1}BP_{2^{n+1}} $
by
\begin{align*}
x_{1,n} = \begin{cases} v_1 & n=0 \\
 v_1^2-4v_1^{-1}v_2 &n=1 \\
 v_1^4-8v_1v_2  & n=2 \\
 x_{n-1}^2 & n\geq 3.
 \end{cases}
\end{align*}
For $s\geq 1$ an odd integer,
the reduction of $x_{1,n}^s$ modulo $2$ is an element of $BP_{2^{n+1}s}/2$ congruent to $v_1^{2^ns}$. Furthermore, 
\[ x^s_{1,n} \in \begin{cases} BP_{2}/2 & \text{$n=0$ and $s\geq 1$} \\
BP_{4}/4 & \text{$n=1$ and $s=1$} \\
BP_{2^{n+1}s}/2^{n+2}& \text{$n\geq 2$, or $n=1$ and $s\geq 3$} 
\end{cases}\]
are comodule primitives. 

Let
\[\bock{n}\colon \Ext^{0,t}_{BP_\ast BP}(BP_\ast,BP_\ast/2^{n}) \to \Ext^{1,t}_{BP_\ast BP}(BP_\ast,BP_\ast) \] 
be the connecting Bockstein homomorphism 
associated to the short exact sequence 
\begin{align*}
\xymatrix@C=25pt{
0 \ar[r] &  BP_* \ar[r]^-{\times 2^{n}} & BP_* \ar[r] & BP_*/2^{n} \ar[r] & 0 \ .
}
\end{align*}
Keeping the convention that $s\geq 1$ is an odd integer, there are classes 
\[\alpha_{i/j} \in \Ext_{BP_*BP}^{1,2i}(BP_*,BP_*)\]
of order $2^j$ defined by
\begin{align*}
\alpha_{s/1} &= \bock{1}(x_{1,0}^s),  \\ 
\alpha_{2/2} &= \bock{2}(x_{1,1}), 
\end{align*}
and
\begin{align*}
 \alpha_{2^{n}s/(n+2)} &=  \bock{2^{n+2}}(x_{1,n}^s) 
\end{align*}
for $n\geq 2$ and $s \geq 1$, or for $n=1$ and $s\geq 3$. We usually abbreviate $\alpha_{i} = \alpha_{i/1}$.

Note that $\alpha_1 \alpha_{2/2}=0$ and otherwise
\[\alpha_1^k \alpha_{i/j} \neq 0\]
for all $k\geq 0$. There are differentials
\begin{align*}
d_3(\alpha_{i/j}) &= \begin{cases} \alpha_1^4 & \text{$i=j=1$} \\
\alpha_1^3 \alpha_{4k+1} & \text{$i=4k+3$ and $j=1$} \\
\alpha_1^3 \alpha_{2^ns/n+2} & \text{$i=2^ns+2$ and $j=3$}.
\end{cases}
\end{align*}
We obtain the pattern in \fullref{fig:BP}, which is also in Table 2 of \cite{ravnovice}.
\begin{figure}
\center
\includegraphics[width=\textwidth]{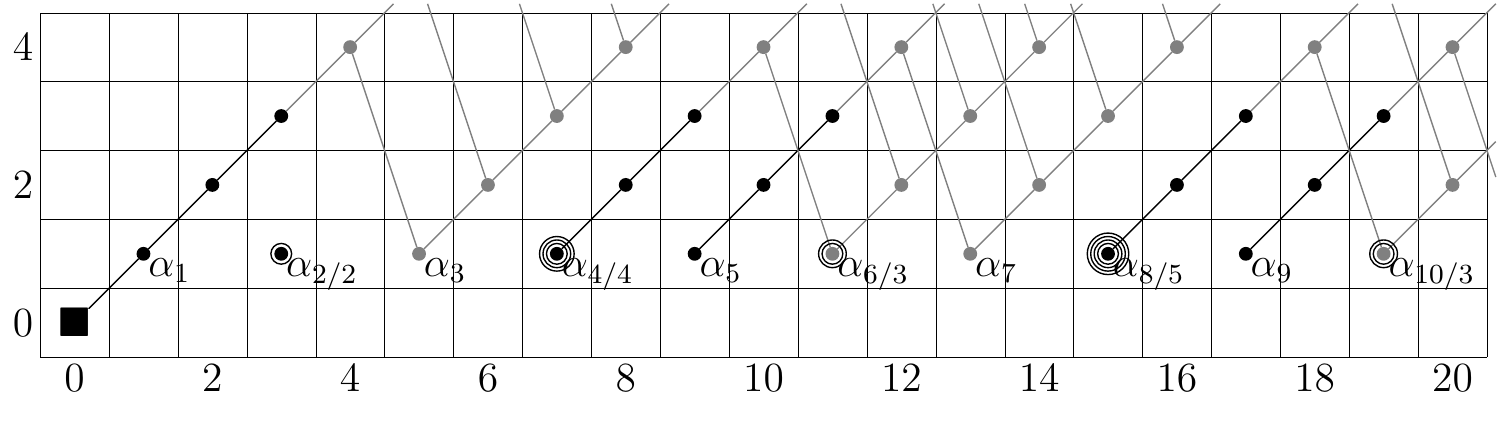}
\includegraphics[width=\textwidth]{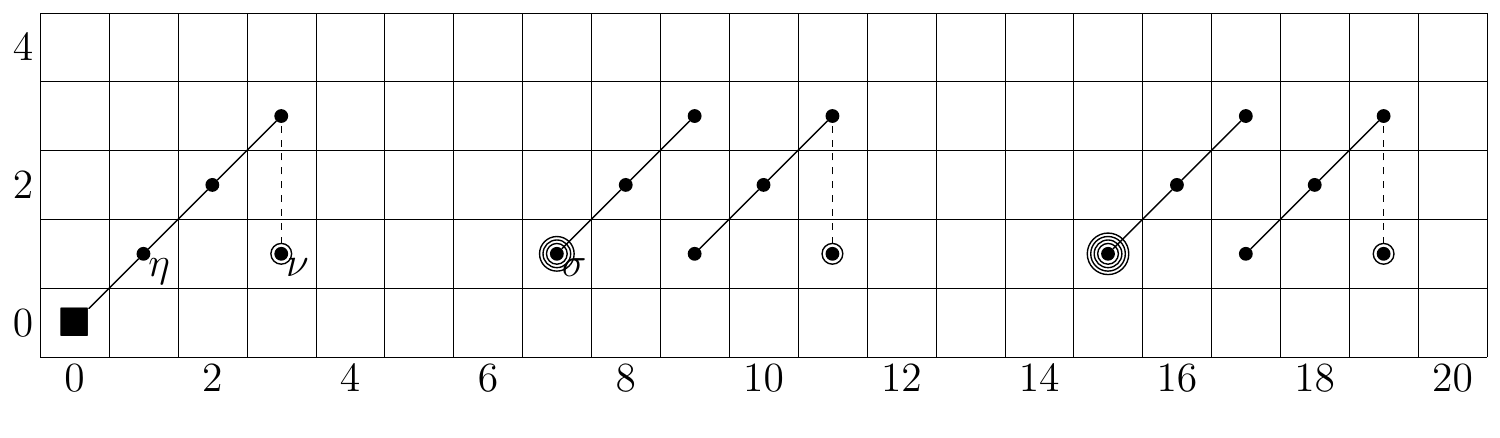}
\captionsetup{width=\textwidth}
\caption{The $\alpha$-family in the $E_2$ (top) and $E_{\infty}$ (bottom) pages of the Adams-Novikov Spectral Sequence. Here, a $\blacksquare$ denotes a copy of $\Z_2$, a $\bullet$ denotes a copy of $\Z/2$, a 
\circled{$\bullet$} a copy of $\Z/4$ and so on. Dashed lines denote exotic multiplications by $2$.}
\label{fig:BP}
\end{figure}

\section{Subgroups of $\G_2$ and the algebraic duality spectral sequence}\label{sec:tools}
Before turning to the computation of the $\alpha$-family in the $K(2)$-local sphere, we recall some of the tools used in the computation. This will be brief, but we refer the reader to \cite{beaudryresolution, beaudrytowards, BGH} where these techniques were explained in great detail.

We let $K(2)$ refer to the $2$-periodic Morava $K$-theory spectrum whose formal group law is that of the super-singular elliptic curve defined over $\F_4$
with Weierstrass equation 
\[C_0 : y^2+y=x^3.\]
The homotopy groups of $K(2)$ are given by
\[K(2)_*= \F_4[u^{\pm 1}]\]
for $u$ in degree $-2$. We let $E=E_2$ be the associated Morava $E$-theory constructed in \cite[Section 2]{beaudrytowards}, chosen so that the formal group law of $E$ is that of the universal deformation of $C_0$ with Weierstrass equation
\[C : y^2+ 3u_1xy+(u_1^3-1)y=x^3.\]
Its homotopy groups are
\[ E_* = \W[\![ u_1]\!][u^{\pm1}]\]
where $u_1$ is in $E_0$ and $u$ is in $E_{-2}$. Here $\W = W(\F_4)$ is the ring of Witt vectors on $\F_4$. We choose a primitive third root of unity $\omega$ and note that $\W \cong \Z_2[\omega]/(1+\omega + \omega^2)$. This is a complete local ring with residue field $\F_4$. In fact, it is the ring of integers in an unramified extension of degree $2$ of $\Q_p$.
The Galois group $\Gal = \Gal(\F_4/\F_2)$, whose generator we denote by $\sigma$, acts on $\W$ by the $\Z_2$-linear map determined by $\omega^{\sigma} = \omega^2$. Further, the Teichm\"uller lifts give a natural embedding of $\F_4^{\times} \subseteq \W^{\times}$. 

We let $\SS_2$ be the group of automorphisms of the formal group law of $K(2)$. The group $\SS_2$ is isomorphic to the units in a maximal order $\mathcal{O}$ of a division algebra of dimension $4$ over $\Q_2$ and Hasse invariant $1/2$. A presentation for $\mathcal{O}$ is given by
\[ \mathcal{O}  \cong \W\langle T\rangle / (T^2=-2, aT=Ta^{\sigma}), \ \ a\in \W.\]
It follows that an element of $\gamma \in \SS_2$ can be written as power series
\[ \gamma= \sum_{i \geq 0} a_i(\gamma) T^i\]
where the elements $a_i(\gamma) \in\W$ satisfy $a_i(\gamma)^4-a_i(\gamma)=0$ and $a_0(\gamma) \neq 0$. The Galois group $\Gal = \Gal(\F_4/\F_2)$ acts on $\mathcal{O}$ via its action on $\W$, fixing $T$. We let $\GG_2$ be the extension of $\SS_2$ by $\Gal$, so that
\[\GG_2 = \SS_2 \rtimes \Gal.\]

The right action of $\SS_2$ on $\mathcal{O}$ gives rise to a representation $\SS_2 \to GL_2(\W)$ whose determinant restricts to a homomorphism
\[  \det \colon \SS_2 \to \Z_2^{\times} . \]
We can extend the determinant to $\GG_2$ by $\det(x, \sigma) = \det(x)$. The determinant composed with the projection to $\Z_2^{\times}/(\pm 1) \cong \Z_2$ defines a homomorphism of $\GG_2$ onto $\ZZ_2$. For any subgroup $G \subseteq \GG_2$, we let $G^1$ be the kernel of this composite. If $G$ is $\SS_2$ or $\GG_2$, this is a split surjection so that
\begin{align}\label{eq:SS1}
\SS_2 &\cong \SS_2^1 \rtimes \ZZ_2  & \GG_2 \cong \GG^1_2 \rtimes \ZZ_2 
\end{align}
We will use the map $\ZZ_2 \to \GG_2$ which send a chosen generator if $\Z_2$ to $\pi = 1+2\omega \in \W^{\times}$ as the preferred splitting. 

The group $\SS_2$ has the following important subgroups. First, it has a unique conjugacy class of maximal finite subgroup. A representative can be chosen to be the image of the automorphisms of the super-singular curve $\Aut(C_0)$, which we will denote by $G_{24}$. It is the semi-direct product of a quaternion group with the natural copy of $\F_4^{\times}$ in $\SS_2$. The group $C_6 = (\pm 1) \times \F_4^{\times}$ is a subgroup of $G_{24}$. Note that the torsion is contained in $\SS^1_2$ as $\Z_2$ is torsion free. So these are in fact subgroups of $\SS^1_2$. However, we note that in $\SS^1_2$ the groups $G_{24}$ and $G_{24}' = \pi G_{24} \pi^{-1}$ are not conjugate ($\pi \not\in \SS^1_2$). Finally, the Galois group acts on these finite subgroups and they can all be extended to corresponding subgroups of $\G_2$. The maximal finite subgroup of $\G_2$ is denoted by 
\[G_{48} \cong G_{24} \rtimes \Gal .\]

Next, we turn to the computational tools. 
For finite spectra $X$, 
\[X_{K(2)} \simeq E^{h\GG_2}\smsh X \simeq (E\smsh X)^{h\G_2} \] 
and there is a spectral sequence
\[E_2^{s,t} = H^s(\GG_2, E_tX) \Longrightarrow \pi_{t-s}X_{K(2)}  \]
where, here and everywhere, we mean the continuous cohomology groups. Analyzing the $E_2$-term of this spectral sequence is difficult, so we often start by studying the cohomology of the subgroup $\SS_2^1$. We have an extremely concrete tool to compute the group cohomology of $\SS_2^1$, a spectral sequence called the \emph{algebraic duality spectral sequence} (ADSS), which we describe here.

For a graded profinite $\Z_2[\![\SS_2^1]\!]$-module $M$ (a typical example is $M = E_*X$), the algebraic duality spectral sequence for $M$ is a first quadrant spectral sequence:
\begin{align}\label{eq:ADSS}
E_1^{p,q,t} =E_1^{p,q,t}(M)  \cong H^q(F_p, M_t) \Longrightarrow H^{p+q}(\SS_2^1, M_t)\end{align}
with differentials $d_r \colon E_r^{p,q,t} \to E_r^{p+r,q-r+1,t}$, where $F_0 = G_{24}$, $F_1=F_2 = C_6$ and $F_3=G_{24}'$. We may omit the internal grading $t$ from the notation.

The spectral sequence has an edge homomorphism
\[H^{p}(H^0(F_{\bullet}, M_t)  ,d_1)\to H^{p}(\SS_2^1,M_t),\]
where $H^{p}(H^0(F_{\bullet}, M_t)  ,d_1)$ is the cohomology of the complex
\begin{equation}\label{eq:complex}
 \xymatrix{ 0 \ar[r] & H^0(F_{0}, M_t)  \ar[r]^-{d_1} & H^0(F_{1}, M_t)  \ar[r]^-{d_1}&H^0(F_{2}, M_t) \ar[r]^-{d_1}& H^0(F_{3}, M_t)  \ar[r] & 0,}
 \end{equation}
and $H^0(F_p, M_t) = E_1^{p,0,t} \cong M_t^{F_p}$. 

Central to the computations of this paper is the differential $d_1 \colon E_1^{0,0,t} \to E_1^{1,0,t}$, which we describe here. 
There is an element $\alpha \in \W^{\times} \subseteq \mathbb{S}_2$ which is defined so that $\alpha = 1+2\omega \mod (4)$ and $\det(\alpha) =-1$.\footnote{
At this point, we run into a conflict of notation. In the current trend of $K(2)$-local computations at $p=2$, the element named $\alpha$ plays a crucial and well-established role. We will keep the name, as any element of the $\alpha$-family has a subscript and this should make it easy to avoid confusion. }
It is shown in Theorem 1.1.1 of \cite{beaudrytowards} that the differential is induced the action of $1-\alpha$:
\[d_1=1-\alpha \colon H^0(F_0, E_*) \to H^0(F_1, E_*). \]

To compute with this spectral sequence, we will also need information about the cohomology $H^*(G_{24}, E_*)$ (which is isomorphic to $H^*(G_{24}', E_*) $) and of $H^*(C_6, E_*)$. This is all well-known, but nicely presented in Section 2 of \cite{BobkovaGoerss}. So we refer to that paper for the information we need.

Finally, for any subgroup $G$ of $\GG_2$, there is always a comparison diagram
\begin{equation*}
\xymatrix{
\Ext^{\ast,\ast}_{BP_\ast BP}(BP_\ast,BP_\ast X) \ar[d] \ar@{=>}[r] & \pi_\ast X \ar[d] \\
H^\ast(\GG_2,E_\ast X) \ar[d] \ar@{=>}[r] & \pi_\ast X_{K(2)} \ar[d] \\
H^\ast(G,E_\ast X) \ar@{=>}[r] & \pi_\ast (E^{hG}\smsh X).
}
\end{equation*}
To detect the $\alpha$-family in $\pi_\ast E^{hG}$, one studies the faith of the $\alpha$-family in $\Ext^{\ast,\ast}_{BP_\ast BP}(BP_\ast,BP_\ast)$ under the vertical maps.

\section{The $\alpha$-family in the $K(2)$-local sphere}
We finally turn to the computation of the $\alpha$-family in the $K(2)$-local sphere. The approach is as follows. We will identify the image of the $\alpha$-family under the map
\[ \Ext_{BP_*BP}^{*,*}(BP_*, BP_*) \to H^*(\GG_2, E_*). \]
In particular, we will show that all of the non-trivial classes $\alpha_1^k \alpha_{i/j}$ map non-trivially. For filtration reasons, this will imply that any class from the topological $\alpha$-family in $\pi_*S$ maps non-trivially to $\pi_*E^{h\GG_2}$.

We will need the following generalization of \cite[Proposition 3.2.2]{BGH}, which allows us to identify classes detecting the $\alpha$-family in the cohomology of certain closed subgroups of $\mathbb{G}_n$.
Its proof is completely analogous and is omitted here.
\begin{prop}\label{prop:detect}
Let $E=E_n$ and $H \subseteq \G_n$ be a closed subgroup. Let $R = H^0(H, E_0)$. Let $y_{i/j} \in E_{2i}/2^j$ be a class so that
\begin{enumerate}[(1)]
\item $y_{i/j} \equiv v_1^{2^i}$ modulo $2$,
\item $y_{i/j}$ is invariant under the action of $H$, and
\item $H^0(H, E_{2i}/2)$ is a cyclic $R$-module generated by $v_1^{i}$.
\end{enumerate}
Then, up to multiplication by a unit in $R$, the image of $\alpha_{i/j} \in \pi_{2i-1}E^{hH}$ is detected in the spectral sequence
\[H^s(H, E_t) \Longrightarrow \pi_{t-s} E^{hH}\]
by the class $\bock{j}(y_{i/j}) \in H^1(H, E_{2i})$.
\end{prop}
One of the consequences of Theorem 1.2.2 of \cite{beaudrytowards} is the following lemma.
\begin{lem}
Let $H$ be a closed subgroup of $\GG_2$ which contains $\mathbb{S}_2^1$. Then
\[H^0(H, E_*/2) \cong \FF_4[v_1] .\]
\end{lem}
\noindent
Therefore, in any positive degree, $H$ satisfies the condition of \fullref{prop:detect} provided that it contains $\mathbb{S}_2^1$. So, to apply \fullref{prop:detect}, we must identify candidates for the classes $y_{i/j}$.

To construct these classes, recall that there are classical $G_{48}$-invariants in $E_0$ associated to the curve $C$, which play a key role in computations at $n=p=2$. Specifically, letting $v_1  = u_1 u^{-1}$ and $v_2=u^{-3}$ the following are invariant for the action of $G_{48}$:
\begin{align*}
\Delta &= 27v_2^3(v_1^3 - v_2)^3 \\
c_4&=9 (v_1^4+8v_1v_2)  \\
c_6&=27 \left(v_1^6-20 v_1^3v_2-8v_2^2\right) \\
j &= c_4^3 \Delta^{-1}.
\end{align*}

A few elements in the higher cohomology $H^*(G_{48}, E_*)$ will also appear in the computation. Namely, there are elements
\begin{align*}
\eta \in H^1(G_{48}, E_2) & &  \nu \in H^1(G_{48}, E_4) &   & \mu  \in H^1(G_{48}, E_6).
\end{align*} 
The classes $\eta$ and $\nu$ are chosen to be the images of $\alpha_1$ and $\alpha_{2/2}$ under the map
\[\Ext_{BP_*BP}^{*,*}(BP_*,BP_*) \to H^*(G_{48}, E_*).  \]
We choose the class $\mu$ to be the image of $\alpha_{3}$. This will be discussed in the proof of \fullref{thm:sood}. It has the property that $\mu = \eta v_1^2$ modulo $(2)$ and $\eta \Delta^{-1}c_6c_4^2$ is a unit multiple of $j \mu$. 

Note that $H^*(G_{24}, E_*) \cong H^*(G_{48}, E_*) \otimes_{\Z_p}\W$. The restriction $H^*(G_{48}, E_*) \to H^*(G_{24}, E_*)$ is the inclusion of fixed point under the action of the Galois group on the right factor of $\W$. For any element in the cohomology of $G_{48}$, we denote its restriction in the cohomology of $G_{24}$ by the same name.

We will prove the following result.
\begin{prop}\label{prop:invariants}
Let $s \geq 1$ be odd. Then, for the action of $\GG_2$,
\begin{enumerate}[(a)]
\item $v_1^s \in E_2$ is an invariant modulo $2$,
\item $v_1^2 \in E_4$ is an invariant modulo $4$,
\item $c_4^{n}\in E_{8n}$ is an invariant modulo $2^{k+4}$ for $n=2^ks $ where $k\geq 0$, and
\item $c_6 c_4^{n} \in E_{8n+12}$ is invariant modulo $8$ for $n\geq 0$.
\end{enumerate}
\end{prop}
This motivates the following definition, where $s\geq 1$ is odd,
\begin{align}\label{eq:yij}
y_{i/j} = \begin{cases}  v_1^s & i=s, j=1, \\
v_1^2 & i=j=2 , \\
c_4^{2^k s} & i = 2^{k+2}s, j=k+4, \\
c_6 c_4^{(s-3)/2} & i=2s, j=3, \ s\neq 1 .
 \end{cases}
\end{align}
Note that in the last two cases of \eqref{eq:yij}  (for $i = 2^{k+2}s$ and $j=k+4$, or $ i=2s$, $j=3$, and $s\neq 1$), the element $y_{i/j} \in H^0(F_0, E_{2i})$ since $F_0=G_{24} \subseteq G_{48}$ and both $c_4$ and $c_6$ are invariant for the action of $G_{48}$. 

Following the outline of \fullref{prop:detect}, we must compute $\bock{j}(y_{i/j})$. We get specific and do this for the group $\SS^1_2$ defined in \eqref{eq:SS1} by using the algebraic duality spectral sequence (ADSS) of \eqref{eq:ADSS}. The part of the ADSS relevant for our computations is depicted in \fullref{fig:ADSSE1}.
Lemma 7.1.2 of \cite{BGH} gives a methods for computing the Bockstein $\bock{n}$ of certain elements for the spectral sequence of a double complex which is particularly suited to the ADSS. Combined with \fullref{prop:detect}, it has the following immediate consequence.
\begin{thm}\label{prop:ADSS}
Let $(H^0(F_{\bullet},E_{t}),d_1)$ be the complex of \eqref{eq:complex}.
Let $s\geq 1$ be an odd integer. Let 
\begin{enumerate}[(a)]
\item $i = 2^{k+2}s$ and $j=k+4$, or 
\item $ i=2s$, $j=3$, and $s\neq 1$. 
\end{enumerate}
Then, up to multiplication by a unit in $\W$, $\alpha_{i/j} \in H^1(\mathbb{S}_2^1, E_{2i})$ is detected by
the image of the class
\[\Big[\frac{d_1(y_{i/j})}{2^j}\Big] \in H^{1}(H^0(F_1,E_{2i}),d_1)\]
under the edge homomorphism
\[ H^0(F_p,E_{2i} ) \longrightarrow E_{\infty}^{p,0,2i} \subseteq H^p(\SS_2^1, E_{2i}).\]
\end{thm}

To prove \fullref{prop:invariants} and thus apply \fullref{prop:ADSS}, we will need some information about the action of $\mathbb{S}_2$ on $c_4$ and $c_6$ which we record now.
\begin{prop}\label{prop:act}
Let $\gamma= 1+a_2(\gamma)T^2 \mod T^3$ in $\mathbb{S}_2$. Then
\[\gamma_*(c_4) \equiv c_4 + 16(a_2(\gamma)+a_2(\gamma)^2)v_1v_2 \mod (32, 16 u_1^2) \]
and
\[\gamma_*(c_6)  \equiv c_6+ 8(a_2(\gamma)+a_2(\gamma)^2)v_1^3v_2 \mod (16, 8u_1^4).\]
\end{prop}
\begin{proof}
The first claim is  Lemma 5.2.2 of \cite{beaudrytowards}. To prove the second claim, we proceed as in the proof of this lemma. From (3.3.1) of \cite{beaudrytowards}, we have that
\begin{align}\label{eqn:formulasact}
\gamma_*(u) &= t_0(\gamma) u & \gamma_*(u_1) &= t_0(\gamma) u_1 +\frac{2}{3} \frac{t_1(\gamma)}{t_0(\gamma)}
\end{align}
where
\begin{align}\label{eqn:t0t1}
t_0(\gamma) &\equiv 1+2a_2(\gamma) \mod (2,u_1)^2, & t_1 &\equiv a_2(\gamma)^2 u_1 \mod (2, u_1^2).
\end{align}
We abbreviate by letting $t_i = t_i(\gamma)$ for $i=0,1$ and $a_2=a_2(\gamma)$.

From \eqref{eqn:formulasact}, we deduce that, modulo $(16)$
\begin{align*}
c_6-\gamma_*(c_6) 
&\equiv 4u^{-6} t_0^{-6} \left(u_1^3t_0^6 + u_1^3t_0^2(3 t_0+3u_1 t_1^2 + u_1^2t_1 t_0^2)+2 (u_1^2 t_1 t_0 + t_0^6+1)\right).
\end{align*}
By Proposition 6.3.3 of \cite{beaudrytowards},
\[ t_0^4  \equiv t_0 + u_1t_1^2+ u_1^2t_1t_0^2   \mod (2)\]
so that
\begin{align*}
c_6-\gamma_*(c_6) \equiv 0 \mod 8.
\end{align*}
To compute the leading term, we consider $c_6-\gamma_*(c_6) $ modulo $(16, u_1^4)$. Using \eqref{eqn:t0t1}, we have that, modulo $(16,u_1^4) $
\begin{align*}
c_6-\gamma_*(c_6)  &\equiv 4u^{-6} t_0^{-6} \left(u_1^3(t_0^6 + 3t_0^3) +2 u_1^2 t_1 t_0 + 2(t_0^6+1)\right) \\
& \equiv   8u_1^3u^{-6}  \left(a_2+ a_2^2  \right).
\end{align*}
In the last line, we used the fact that $t_0^{6} \equiv 1 \mod (2,u_1^4)$ and also modulo $(4,u_1)$. 
\end{proof}

\begin{proof}[Proof of \fullref{prop:invariants}]
Parts (a) and (b) are immediate since $v_1$ is invariant modulo $2$. \fullref{prop:act} shows that $c_4$ and $c_6$ are invariant under the action of $\alpha$ and $\pi$ modulo $16$ and $8$ respectively. Since $c_4$ and $c_6$ are already invariant under the action of $G_{48}$ and $\mathbb{G}_2$ is topologically generated by $G_{48}$, $\alpha$ and $\pi$, parts (c) and (d) follow by taking appropriate powers. \end{proof}

To apply \fullref{prop:ADSS}, we will prove something slightly more general: We will completely compute the differential
\[d_1 \colon E_1^{0,0,*} \to E_1^{1,0,*}.\]
We first identify $E_1^{0,0,*} \cong E_*^{G_{24}} \cong H^0(G_{24},E_*)$ and $E_1^{1,0,*} \cong E_*^{C_6}  \cong H^0(C_{6},E_*)$ more explicitly than we have done so far.

For example, from Section 2 and 3 of \cite{BobkovaGoerss}, we have isomorphisms
\begin{align*}
H^0(G_{24}, E_*)\cong \WW[\![j]\!][c_4, c_6, \Delta^{\pm1}]/(  c_4^3- c_6^2=  (12)^3\Delta, c_4^3=\Delta j).
 \end{align*}
It follows that the elements
\[ \{ c_6^{\epsilon} c_4^{m} \Delta^n \mid m \geq 0, \ \epsilon = 0,1, \ n \in \ZZ \}  \]
form a set of topological $\WW$-module generators, so that, in the category of profinite graded $\WW$-modules,
\[H^0(G_{24}, E_*)\cong 
 \prod_{\substack{n, m\in \ZZ,
m\geq 0 \\
\epsilon = 0,1}}
 \WW\{c_6^{\epsilon} c_4^{m} \Delta^n \}. \]
There is also an isomorphim
\[H^0(C_{6}, E_*) \cong \WW[\![u_1^3]\!][v_1^2, v_1v_2, v_2^{\pm 2}]/\sim \]
where $\sim$ is the ideal
\[ (\cvone^3 - \cvtwo \cj^2, \cvonevtwo^2- \cvone \cvtwo, \cvone \cvonevtwo - \cj \cvtwo).\]
Therefore, a basis of topological $\WW$-module generators for $H^0(C_{6}, E_*) $ is given by
\[ \{  \cvonevtwo^{\epsilon} \cvone^{m}    \cvtwo^{n} \mid m \geq 0, \ \epsilon = 0,1, \ n \in \ZZ \}  \]
and, in the category of profinite graded $\WW$-modules,
\[H^0(C_{6}, E_*)\cong \prod_{\substack{n, m\in \ZZ,
m\geq 0\\
\epsilon = 0,1}}
 \WW\{\cvonevtwo^{\epsilon} \cvone^{m}  \cvtwo^{n} \}. \]
 
We are now ready to compute $d_1$ explicitly. We note that this result is intimately related to Propositions 8.1 and 8.2 \cite{MR}.
\begin{prop}\label{prop:d1diff}
The differential $d_1 \colon E_1^{0,0} \to E_1^{1,0}$ is determined by the following information:
\begin{enumerate}[(a)]
\item For $n,m \in\ZZ$ of the form $n=2^k(2t+1)$, $m\geq 0$ and for $\epsilon = 0,1$,
\begin{align*}
d_1(c_6^{\epsilon}c_4^m\Delta^n) &\equiv \cvone^{3\cdot 2^{k}+2m+ 3\epsilon}\cvtwo^{2^{k}(1+4t)}  \mod (2,  v_1^{9\cdot 2^k+4m +6\epsilon}) 
\end{align*}
and $d_1(\Delta^0) = 0$.
\item For $n\in \ZZ$ of the form $n=2^k(2t+1)$, $n\geq 1$,
\begin{align*}
d_1(c_4^{n}) & \equiv  2^{k+4}  \cvonevtwo  \cvone^{2 (n-1)} \mod (2^{k+5}, 2^{k+4}v_1^{4(n-1)+2}).
\end{align*}
\item For $n\in \ZZ$, $n\geq 0$ of the form $n=2^k(2t+1)$ or for $n=0$,
\begin{align*}
d_1(c_6c_4^{n}) & \equiv  8 \cvonevtwo \cvone^{2n+1}  \mod (16, 8 v_1^{4n+4}).
\end{align*}
\end{enumerate}
\end{prop}
\begin{proof}
This differential is given by the action of $1-\alpha$. Further, $\alpha \equiv 1+ \omega T^2$ modulo $T^4$ for $\omega$ a primitive third root of unity. Hence, $a_2(\alpha) + a_2(\alpha)^2 = -1$.

The claim (a) is an immediate consequence of Proposition 5.1.1 of \cite{beaudrytowards}, which states that
\[{\alpha}_*(\Delta^n) \equiv \Delta^n+ v_1^{6\cdot 2^k}v_2^{2^{k+1}(4t+1)} \mod (2, u_1^{9\cdot 2^k}),\]
using the fact that $c_4 \equiv v_1^4$ and $c_6 \equiv v_1^6$ modulo $2$.

To prove (b), from \fullref{prop:act}, using the fact that $c_4 \equiv v_1^4$ modulo $2$, we deduce that
\begin{align*}\alpha_*(c_4^{2^k})
& \equiv (c_4^{2^k} + 2^{k+4}v_1^{4(2^k-1)}v_1v_2) \mod (2^{k+5}, 2^{k+4}u_1^{4(2^k-1)+2}).
\end{align*}
Hence,
\begin{align*}
\alpha_*(c_4^{n}) 
& \equiv c_4^n + 2^{k+4} v_1^{4 (n-1)} v_1v_2 \mod (2^{k+5}, 2^{k+4}u_1^{4(n-1)+2}) .
\end{align*}

Similarly, to prove (c), using that $\alpha_*(c_4) \equiv c_4$ modulo $(16)$ we have
\begin{align*}
\alpha_*(c_6c_4^{n}) & \equiv \alpha_*(c_6)c_4^n \mod (16) \\
 & \equiv  c_6c_4^n +8v_1^{4n+3}v_2  \mod (16, 8 v_1^{4n+4}).\qedhere
\end{align*}
\end{proof}

\begin{rem}
For $(\epsilon, a,b)$ such that $\epsilon = 0,1$, $a \geq 0$, and $b \in \ZZ$,
we define elements $b_{\epsilon,a,b}$ in $E_1^{1,0,t} \cong H^0(C_6, E_t)$ for $t=8\epsilon+4a+12b$ that satisfy
\[ b_{\epsilon,a,b} = \cvonevtwo^{\epsilon} \cvone^{a} \cvtwo^{b} + \ldots  \]
 as follows:
\begin{enumerate}[(a)]
\item For $n,m \in\ZZ$ of the form $n=2^k(2t+1)$, $m\geq 0$ and for $\epsilon = 0,1$,
\begin{align*}
b_{0,3\cdot 2^{k}+2m+ 3\epsilon, 2^{k}(1+4t)}&= d_1(c_6^{\epsilon}c_4^m\Delta^n) 
\end{align*}
\item For $n\in \ZZ$ of the form $n=2^k(2t+1)$, $n\geq 1$,
\begin{align*}
b_{1,2 (n-1),0} &= \frac{d_1(c_4^{n})}{2^{k+4}}.
\end{align*}
\item For $n\in \ZZ$, $n\geq 1$ of the form $n=2^k(2t+1)$ or for $n=0$,
\begin{align*}
b_{1, 2n+1,0} &= \frac{d_1(c_6c_4^{n})}{8}.
\end{align*}
\item
In all other cases, 
\begin{align*}
b_{\epsilon, a,b} &= \cvonevtwo^{\epsilon} \cvone^{a} \cvtwo^{b}.
\end{align*}
\end{enumerate}
Although we will not refer to all of the elements $b_{\epsilon,a,b}$ defined above, it will be useful to have fixed name for them in future computations.
\end{rem}

We now give some consequences of \fullref{prop:d1diff}. We start with an immediate corollary:
\begin{thm}\label{thm:ADSSE}
In the ADSS
\[E_1^{p,q,t} \cong H^q(F_p, E_t) \Longrightarrow H^{p+q}(\SS_2^1, E_t)\]
there is an isomorphism
\[E_2^{0,0,0} \cong E_{\infty}^{0,0,0} \cong \WW\{\la_0\}\]
where $\la_0$ is the unit in $E_1^{0,0,0} \cong H^0(G_{24}, E_0)$. Further, $E_2^{0,0,t}=0$ if $t\neq 0$. 
For $r\geq 0$, the classes $b_{1,r,0}$ are in the kernel of $d_1 \colon E_1^{1,0,t} \to E_1^{2,0,t}$ and detect classes in $E_{\infty}^{1,0,t}$ of degree $t=8+4r$. These classes have order $8$ if $r =2n+1$ and $n\geq 0$. They have order $2^{k+4}$ if $r=2n$ for $n= 2^ks-1$, $s \geq 1$ odd, and $k\geq 0$. 
\end{thm}

\begin{rem}\label{rem:justify-abuse} 
Since the edge homomorphism of the ADSS has the form
\[H^0(F_p,E_* ) \longrightarrow E_{\infty}^{p,0,*} \subseteq H^p(\SS_2^1, E_*),\]
even if the generators $b_{1,n,0}$ are strictly speaking elements of $E_{\infty}^{p,0,*} $, they represent unique elements in the cohomology of $\SS_2^1$, and hence, we can write $b_{1,n,0} \in H^*(\SS^1_2, E_*)$ without any ambiguity.
\end{rem}

As an immediate consequence of \fullref{thm:ADSSE}, we have the following result, which was already proved in \cite{BGH}:
\begin{cor}\label{cor:fixpoints}
The inclusion $\Z_p \to E_0$ induces an isomorphism
\begin{align*}
H^0(\SS_2^1, E_*) &\cong H^0(\SS^1_2, E_0) \cong \W \\
H^0(\GG_2^1, E_*) &\cong  H^0(\GG_2^1, E_0) \cong \Z_p \\
H^0(\GG_2, E_*) &\cong  H^0(\GG_2, E_0) \cong \Z_p .
\end{align*}
\end{cor}
\begin{proof}
\fullref{thm:ADSSE} implies that $H^0(\SS^1_2, E_*) \cong \W$ and since $H^0(\SS^1_2, E_*) \cong H^0(\GG^1_2, E_*) \otimes_{\Z_p}\W$ with the natural action of $\Gal$ on $\W$ (see \cite[Lemma 1.24]{BobkovaGoerss}) the result follows for $\GG_2^1$. The fixed points for $\GG_2$ include in those for $\GG_2^1$ and contain the image of $\Z_p \subseteq E_0$.
\end{proof}

The next three results are depicted in \fullref{fig:ADSS}.
\begin{cor}
Up to multiplication by a unit in $\W$,
\begin{enumerate}[(a)]
\item $b_{1,2n+1,0}$ for $n\geq 0$ detects $\alpha_{(4n+6)/3}$, and
\item $b_{1,2n,0}$ for $n=2^k s-1$, $s\geq 1$ odd, $k\geq 0$ detects $\alpha_{2^{k+2}s/(k+4)}$.
\end{enumerate}
\end{cor}
\begin{proof}
This follows from \fullref{prop:ADSS}, using \fullref{thm:ADSSE} and \fullref{cor:fixpoints}.
\end{proof}

We turn to the elements $\alpha_s = \alpha_{s/1}$ where $s\geq 1$ is odd.
\begin{thm}\label{thm:sood}
Let $s\geq 1$ be an odd integer. In the ADSS
\[E_1^{p,q,t} \cong H^q(F_p, E_t) \Longrightarrow H^{p+q}(\SS_2^1, E_t)\]
there is are isomorphisms
\[E_2^{0,1,2s} \cong E_{\infty}^{0,1,2s} \cong  \begin{cases}\FF_4 \{\eta c_6^{\epsilon}c_4^{m}\} & s\neq 3, \  s= 1+6\epsilon +4 m \\
\FF_4\{\mu\} & s=3.
\end{cases}\]
Further $\alpha_s = \alpha_{s/1} \in H^1(\SS_2^1, E_{2s})$ is non-trivial. The edge homomorphisms 
\[H^1(\SS_2^1, E_{2s}) \to E_{\infty}^{0,1,2s} \subseteq H^1(F_0, E_{2s})  \]
are isomorphisms and $\alpha_{s}$ can be identified with its image in $H^1(F_0, E_{2s})$. The element $\alpha_1$ is detected by $\eta$, $\alpha_3$ is detected by $\mu$. If $s\geq 5$, then $s= 1+6\epsilon +4 m$ for some $\epsilon=0,1$ and $m\geq 0$. In this case, $\alpha_{s}$ is detected by $\eta c_6^{\epsilon}c_4^m$.
\end{thm}
\begin{proof}
The associated graded of the ADSS for the group $H^1(\SS^1, E_{2s})$ consists of $E_{\infty}^{0,1,2s}$ and $E_{\infty}^{1,0,2s}$. The latter is a subquotient of $E_{2}^{1,0,2s} \cong E_{2s}^{C_6}$, which is trivial when $s$ is odd. Therefore, $H^1(\SS_2^1, E_{2s}) \cong E_{\infty}^{0,1,2s}$ and the edge homomorphism is an isomorphism. 

Now, note that the reduction modulo $2$ induces isomorphisms $H^1(F_p,E_{2s}) \cong H^1(F_p,E_{2s}/2)$ for $p=0,1$. Further, this isomorphism maps $\eta  c_6^{\epsilon} c_4^m$ to $\eta v_1^{6\epsilon +4m}$. 
So, to compute $E_{\infty}^{0,1,2s}$ we can use the commutative diagram
\[\xymatrix{
0 \ar[r] & H^1(F_0,E_{2s}) \ar[r]^{d_1} \ar[d]_{\cong}&H^1(F_1,E_{2s}) \ar[d]_{\cong}\\
0 \ar[r] &  H^1(F_0,E_{2s}/2) \ar[r]^{d_1} &H^1(F_1,E_{2s}/2) 
}\]
The kernel of $d_1$ for the top row is isomorphic to the kernel of $d_1$ for the bottom row, which was computed in \cite{beaudrytowards} to be generated by $\eta v_1^{s-1} = \eta v_1^{6\epsilon +4m}$ if $s= 1+ 6\epsilon +4m$. Therefore, $E_2^{0,1,2s} \cong \F_4\{ \eta c_6^{\epsilon}c_4^m\}$ as desired. 

Also implied by the extensive computations in \cite{beaudrytowards} is the fact that $H^0(\SS^1_2, E_*) \cong \F_4[v_1]$
with the edge homomorphism
\[\xymatrix{H^0(\SS^1_2, E_{2s}/2) \ar[r]^-{\cong} & E_{\infty}^{0,0,2s} \subseteq H^0(F_0, E_{2s}/2)}  \]
 an isomorphism.
Consider the commutative diagram
\[\xymatrix{H^0(\SS^1_2, E_{2s}/2) \ar[r]^-{\delta_{\SS^1_2}} \ar[d]^{\cong} & H^1(\SS^1_2, E_{2s}) \ar[d]^-{\cong} \\
E_{\infty}^{0,0,2s} \ar[d]^-{\subseteq} & E_{\infty}^{0,1,2s} \ar[d]^-{\subseteq}\\
H^0(F_0, E_{2s}/2) \ar[r]^-{\delta_{F_0}}  & H^1(F_0, E_{2s}) \\
}\]
where $\delta_G$ is the connecting homomorphism for the exact sequence 
\[\xymatrix{0 \ar[r] &  E_* \ar[r]^-{\times 2}  & E_* \ar[r] & E_*/2 \ar[r] & 0. }\]
Since 
\[\delta_{F_0}(v_1^s) =  \begin{cases}
\eta c_6^{\epsilon}c_4^{m} & s\geq 1, \ s\neq 3, \  s= 1+6\epsilon +4 m \\
\mu & s=3 ,
\end{cases} \]
the image of $\alpha_{s}$ is $\eta c_6^{\epsilon}c_4^{m} $ if $s \neq 3$ and $\mu$ if $s=3$. So, the corresponding elements of $E_2^{0,1,2s}$ are permanent cycles in the ADSS. So $E_{\infty}^{0,1,2s} \cong E_2^{0,1,2s}$, generated by the image of $\alpha_{s}$ for $s$ odd.
\end{proof}

It remains to understand the image of $\alpha_{2/2}$.
\begin{thm}
In the ADSS, there  is an isomorphism
\[ E_2^{0,1,4} \cong E_{\infty}^{0,1,4} \cong \W/4\{\nu\} .\]
The edge homomorphism 
\[H^1(\SS_2^1, E_{4}) \to E_{\infty}^{0,1,4} \subseteq H^1(F_0, E_{4})  \]
is an isomorphism and the element $\alpha_{2/2}$ can be identified with its image in $H^1(F_0, E_{4})$, where it is detected by $\nu$. 
\end{thm}
\begin{proof}
The contributions to $H^1(\SS_2^1, E_4)$ in the ADSS consist of $E_{\infty}^{0,1,4}$ and $E_{\infty}^{1,0,4}$. There is an isomorphism
\[E_1^{0,1,4} \cong H^1(F_0, E_{4})  \cong H^1(G_{24}, E_{4})  \cong \W/4\{\nu\}\]
and $\nu$ is so named because it is the image of $\alpha_{2/2}$ under the homomorphism from the ANSS $E_2$-term $\Ext_{BP_*BP}^{1,4}(BP_*, BP_*)$. This map factors through $H^*(\SS_2^1, E_*)$, so $\nu$ must be a permanent cycle in the ADSS. So, all elements of $E_1^{0,1,4} $ persist to $E_{\infty}$. 

We turn our attention to $E_{\infty}^{1,0,4}$ and show that 
\[ E_{\infty}^{1,0,4}(E_*) = E_{2}^{1,0,4}(E_*) \cong H^1( H^0(F_{\bullet}, E_4)) =0.\] 
This implies that the edge homomorphism is an isomorphism, and that $\nu$ lifts uniquely to an element of $H^1(\SS_2^1, E_4)$ where it corresponds to the image of $\alpha_{2/2}$.

We begin with a computation modulo $(2)$. There is a commutative diagram:
\[\xymatrix{
H^0(F_0,E_{4}/2) \ar[r]^{d_1} &H^0(F_1,E_{4}/2) \ar[r]^{d_1}&H^0(F_2,E_{4}/2) \ar[r]^-{d_1} & H^0(F_3,E_4/2)     \\
H^0(F_0,E_0/2) \ar[r]^{d_1}  \ar[u]^-{v_1^2}_-{\cong}   &H^0(F_1,E_0/2) \ar[r]^{d_1}  \ar[u]^-{v_1^2}_-{\cong}  &H^0(F_2,E_0/2)  \ar[u]^-{v_1^2}_-{\cong}\ar[r]^-{d_1}   & H^0(F_3,E_0/2)  \ar[u]^-{v_1^2}_-{\cong}    \\
H^0(F_0,\W/2) \ar[r]^{d_1}  \ar[u]  &H^0(F_1,\W/2) \ar[r]^{d_1}  \ar[u] &H^0(F_2,\W/2)  \ar[u]  \ar[r]^{d_1}  & H^0(F_3,\W/2)  \ar[u]   .
}\]
By Theorem 5.4.1 of \cite{BGH}, the vertical map from the bottom to the middle row induces an isomorphism upon taking cohomology with respect to $d_1$. The cohomology of the bottom row gives a copy of $\F_4$ in each degree, whose generators were called $\la_0$, $\lb_0$, $\lc_0$ $\ld_0$ for $p=0,1,2,3$ respectively. It follows that 
\[ E_2^{p,0,4}(E_*/2)  \cong \begin{cases} \F_4\{ v_1^2 \la_0\}    & p=0 \\
 \F_4\{ v_1^2 \lb_0\}    & p=1.
\end{cases}
\]
Let $\ker(d_1)$ be the kernel of $ H^0(F_1, E_4)  \xra{d_1} H^0(F_2, E_4) $ and $\widetilde{\ker}(d_1)$ that of $H^0(F_1, E_4/2)  \xra{d_1} H^0(F_2, E_4/2) $.
The diagram
\[ \xymatrix{ H^0(F_0, E_4/2)  \ar[d]^-{d_1} \ar[r]^-{\cong}  & \F_4\{v_1^2\la_0\}   \oplus H^0(F_0, E_4)/2 \ar[d]^-{0 \oplus d_1 } \\  
\widetilde{\ker}(d_1) \ar[r]^-{\cong}   & \F_4\{v_1^2\lb_0\}  \oplus \ker(d_1)/(2)
}    \]
commutes. Given the cohomology of the left vertical map, it must be the case that $d_1$ induces an isomorphism
\[d_1 \colon  H^0(F_0, E_4)/2  \xra{\cong} \ker(d_1)/(2). \]

So, we have a commutative diagram
\begin{equation}\label{eq:bigsquare}
\xymatrix{  H^0(F_0, E_4) \ar[r]^-{d_1} \ar[d]^-{\times 2} & \ker(d_1)\ar[d]^-{\times 2}  \ar[r] & H^1(H^0(F_{\bullet}, E_4))   \ar[d]^-{\times 2} \\
H^0(F_0, E_4) \ar[r]^-{d_1} \ar[d] & \ker(d_1)\ar[r] \ar[d]  & H^1(H^0(F_{\bullet}, E_4)) \ar[d]   \\
 H^0(F_0, E_4)/2 \ar[r]^-{d_1}_-{\cong} & \ker(d_1)/(2) \ar[r] & 0   .}\end{equation}
 The left two columns of \eqref{eq:bigsquare} are short exact by definition.
By \fullref{thm:ADSSE} $E_2^{0,0,4} = 0$, so the map $d_1 \colon H^0(F_0, E_4)  \to H^0(F_1, E_4) $ is injective. So the rows of \eqref{eq:bigsquare} are short exact. It follows that the third column is short exact. So multiplication by $2$ is an isomorphism on $ H^1(H^0(F_{\bullet}, E_4))$. Since this is a complete $\Z_2$-module, it must be trivial.
\end{proof}

We now identify the $\alpha$-family in $H^*(\GG_2,E_*)$. We begin with an observation.
\begin{rem}\label{rem:naming}
Recall once more that $H^*(\SS^1_2, E_*) \cong H^*(\GG^1_2, E_*) \otimes_{\Z_2}\W$ and that the restriction
\[H^*(\GG^1_2, E_*)  \cong H^*(\SS^1_2, E_*)^{\Gal} \to  H^*(\SS^1_2, E_*) \]
is an inclusion. We have identified the $\alpha_{i/j}$s in $H^1(\SS^1_2, E_{2i})$ up to a unit in $\W$. We now fix $\alpha_{i/j}$ to be a choice of Galois invariant generator for the $\W$-cyclic subgroup the classes we have found generate. We denote by the same name the unique element in $H^*(\GG^1_2, E_*)$ that maps to it.

On the other hand, the map
\[ H^*(\GG_2, E_*) \to H^*(\GG^1_2, E_*)\]
is not injective. So one must proceed with care. Recall that there is a split exact sequence
\[\xymatrix{  1 \ar[r] & \GG^1_2 \ar[r] &  \GG_2 \ar[r] & \Z_2^{\times}/(\pm1) \ar[r] & 1.} \]
Using the fact that $  \Z_2^{\times}/(\pm1) \cong \Z_2$, we have $\GG_2 \cong \GG^1_2 \rtimes  \Z_2$. We choose $\pi$ to be a topological generator for $\Z_2 \cong \GG_2/\GG^1_2$, which acts on $H^*(\GG^1_2,E_*)$. Let $\mathrm{res} \colon H^*(\GG_2,E_*)  \to H^*(\GG^1_2,E_*)$ be the restriction. From the Lyndon-Hochschild-Serre Spectral Sequence for the group extension, one obtains a long exact sequence
\begin{equation}\label{eq:les}
\xymatrix{\ldots  \ar[r] &  H^*( \GG^1_2 , E_*) \ar[r]^-{\pi -1} & H^*( \GG^1_2 , E_*) \ar[r]^-{\delta} &H^{*+1}( \GG_2 , E_*) \ar[r]^-{\mathrm{res}} & H^{*+1}( \GG^1_2 , E_*) \ar[r]^-{\pi-1} &  \ldots   }\end{equation}
To fully analyze the long exact sequence \eqref{eq:les}, we would need a full computation of $H^*( \GG^1_2 , E_*)$, and control over the action of $\pi$ on $H^*( \GG^1_2 , E_*)$. Neither is available to us at this point. However, in the range of interest for computing the $\alpha$-family, we get lucky. 
\end{rem}

\begin{figure}
\center
\includegraphics[width=\textwidth]{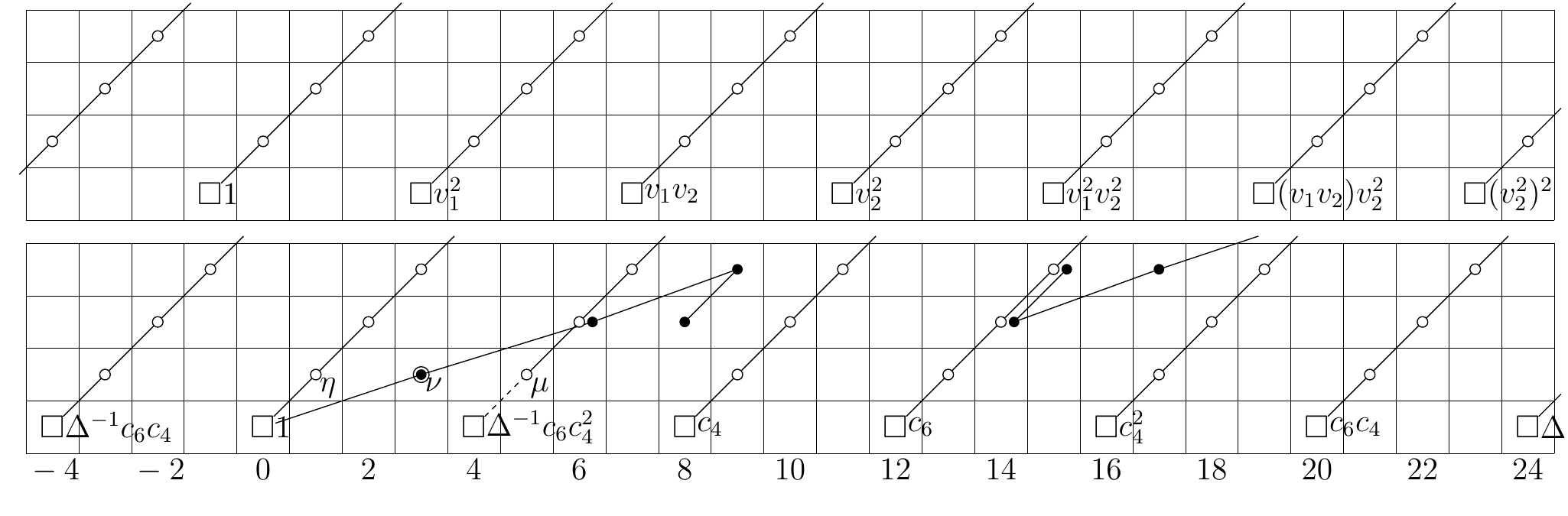}
\captionsetup{width=\textwidth}
\caption{A part of the $E_1$-term of the algebraic duality spectral sequence, $E_1^{p,q,t}=H^q(F_p, E_t)$. The top is $E_1^{1,q,t}$ for $0\leq q \leq 3$, drawn in the $(t-q-1,q)$-plane. The bottom is $E_1^{0,q,t}$ in the same range, drawn in the $(t-q,q)$ plane. A $\Box$ denotes a copy of $\W[\![j]\!]$ if $p=0$ and $\W[\![u_1^3]\!]$ if $p=1$. A $\circ$ denotes a copy of $\F_4[\![j]\!]$ if $p=0$ and $\F_4[\![u_1^3]\!]$ if $p=1$. A $\bullet$ is a copy of $\F_4$ and \circled{$\bullet$} a copy of $\W/4$.
The labels denote the generators as $\W[\![j]\!]$-modules on the $p=0$-line and as $\W[\![u_1^3]\!]$-modules on the $p=1$-line. The lines denote multiplication by $\eta$ and $\nu$. The dashed line indicates that $\eta \Delta^{-1}c_6c_4^2 = j \mu$.
}
\label{fig:ADSSE1}
\end{figure}

\begin{figure}
\center
\includegraphics[width=\textwidth]{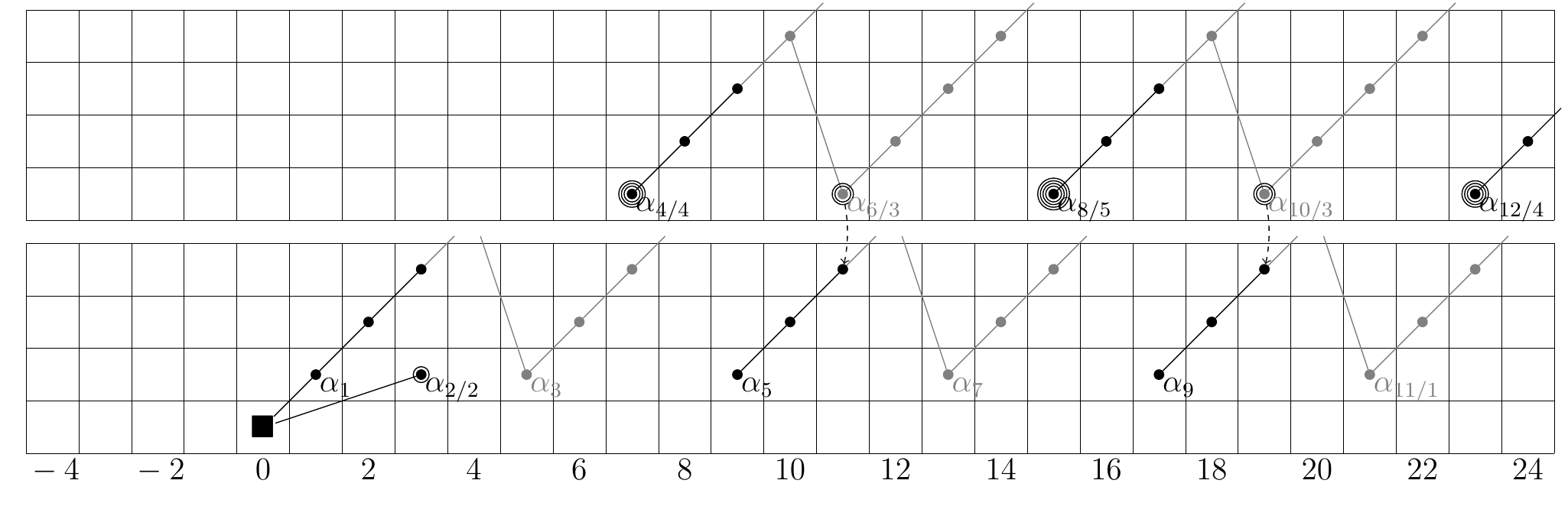}
\captionsetup{width=\textwidth}
\caption{The contribution to the $\alpha$-family in the algebraic duality spectral sequence. The differentials indicate differentials that occur in the homotopy fixed points spectral sequence $H^*(\SS_2^1, E_*) \Longrightarrow \pi_*E^{h\SS_2^1}$ and the dashed arrows indicate exotic extensions on the $E_{\infty}$-term of that spectral sequence.}
\label{fig:ADSS}
\end{figure}

\begin{prop}
The restriction $H^1( \GG_2 , E_t)  \to H^{1}( \GG^1_2 , E_t) $ is injective if $t \neq 0$.
\end{prop}
\begin{proof}
By \fullref{cor:fixpoints}, $H^0(\GG_2^1, E_*) \cong \Z_p$ and the restriction $H^0(\GG_2,E_*) \to H^0(\GG^1_2,E_*)$ is an isomorphism. This is the first map in \eqref{eq:les}, so we get an exact sequence
\[\xymatrix{ 0 \ar[r] &  H^0( \GG^1_2 , E_*)  \ar[r]^-{\delta} & H^1( \GG_2 , E_*) \ar[r] &H^{1}( \GG^1_2 , E_*) \ar[r]^{\pi-1} & H^{1}( \GG^1_2 , E_*)  }\]
The claim follows from the fact that $H^0( \GG^1_2 , E_t)  =0 $ of $t\neq 0$, 
 \end{proof}
 
 \begin{cor}
 There are unique classes in $\alpha_{i/j} \in H^1(\GG_2, E_{2i})$ which map to the same named classes in $H^1(\SS^1_2, E_{2i})$ as described in \fullref{rem:naming}. These are the images of the $\alpha$-family elements under the map from the $E_2$-term of the $BP$-based Adams-Novikov Spectral Sequence.
 \end{cor}

\begin{cor}
The topological $\alpha$-family maps non-trivially in $\pi_*S_{K(2)}$, and in $\pi_*E^{h\SS_2^1}$.
The $\alpha$-family is detected in $H^{*}(\GG_2, E_{*})$ by the classes $\alpha_1^k \alpha_{i/j}$ which support the standard pattern of differentials in the spectral sequence
 \begin{equation}\label{eq:Dsss} H^s(\GG_2, E_t) \Longrightarrow \pi_{t-s}E^{h\G_2}.\end{equation}
\end{cor}
\begin{proof}
The only thing to justify is that there are no differentials killing non-trivial elements of the image of the topological $\alpha$-family. However, the ANSS filtration of the $\alpha$-elements detecting non-trivial elements in homotopy is at most $3$. The first differential in \eqref{eq:Dsss} and the analogue for $\SS_2^1$ is a $d_3$, and the zero line of \eqref{eq:Dsss} consists of the permanent cycles $H^0(\GG_2, E_*) \cong\Z_p$, respectively $H^0(\SS_2^1, E_*) \cong \W$. That the latter are all permanent cycles follows from Section 1.2 of \cite{BobkovaGoerss}.
\end{proof}

\end{document}